\documentclass[12pt]{amsart}
\usepackage{amsmath, amssymb}
\newtheorem{theorem}{\sc Theorem}[section]
\newtheorem{lemma}[theorem]{\sc Lemma}
\newtheorem{proposition}[theorem]{\sc Proposition}

\usepackage{abstract} % Permite maior personalização

\title[Involutory automorphisms]{Involutory automorphisms of finite groups of odd order}

\author[P. Shumyatsky]{Pavel Shumyatsky} 
\address{Pavel Shumyatsky: Department of Mathematics, University of Brasilia, Brasilia, DF, Brazil}
\email{pavel@unb.br}

\author[G. C. de Souza]{Gabriella Cristina de Souza} 
\address{Gabriella Cristina de Souza: Department of Mathematics, University of Brasilia, Brasilia, DF, Brazil}
\email{gabriellacs.math@gmail.com}

\thanks{The work of the first author was supported by FAPDF and CNPq}

\keywords{Finite groups, automorphisms, centralizers}

\subjclass[2020]{ 20D25, 20D45}

\begin{document}

\maketitle

\begin{abstract} \noindent For a finite group $G$ we write $\gamma_\infty(G)$ to denote the nilpotent residual of $G$, that is, the intersection of all terms of the lower central series. The following theorem is proved: 

\noindent Let $G$ be a finite group of odd order admitting an involutory automorphism $\varphi$ such that $G=[G,\varphi]$ and suppose that $\gamma_\infty(C_G(\varphi))$ has order $m$. Then the order of $\gamma_\infty(G')$ is bounded by a function depending only on $m$.

\noindent This complements several earlier results on groups of odd order admitting involutory automorphisms.
\end{abstract}

%%%%%%%%%%%%%%%%%%%%%%%%%%%%%%%%%%%%%%%%%%%%%%%%%%%%%%%%%%%%%%%%%%%
\section{Introduction}

Let $G$ be a finite group admitting an involutory automorphism $\varphi$, that is, an automorphism of order two. It is well known that properties of the fixed-point subgroup $C_G(\varphi)$ exert a strong influence on the structure of $G$.

A classical and elementary example of this phenomenon is the fact that if $C_G(\varphi)=1$, then $G$ is abelian. Hartley and Meixner \cite{hartleyandmeixner} showed that if $C_{G}(\varphi)$ has order $m$, then $G$ possesses a nilpotent subgroup of class at most two  of $m$-bounded index. Throughout this paper, the term ``$m$-bounded'' means ``bounded above by a function depending only on $m$.

It was observed some time ago that if $\varphi$ is an involutory automorphism of a group $G$ of odd order such that $G=[G,\varphi]$, then the structure of the commutator subgroup $G'$ is similar to that of $C_G(\varphi)$.  Asar \cite{asar1981} proved that, if $C_G(\varphi)$ is nilpotent, then $[G,\varphi]'$ is also nilpotent. The first-named author showed that if $[C_G(\varphi):F(C_G(\varphi))]=m$, then the index $[G':F(G')]$ is $m$-bounded \cite{shumyatsky2007}. Here and throughout $F(K)$ stands for the Fitting subgroup of a group $K$. Several other results in this direction were established in \cite{shumyatsky1993,ijac96,arch98,mona05}.
    
Given a finite group $G$, we write $\gamma_\infty(G)$ to denote the nilpotent residual of $G$, that is, the intersection of all terms of the lower central series. The goal of this article is to prove the following result:

\begin{theorem} \label{main}
        Let $G$ be a finite group of odd order admitting an involutory automorphism $\varphi$ such that $G=[G,\varphi]$. If $\gamma_{\infty}(C_G(\varphi))$ has order $m$, then $\gamma_{\infty}(G')$ has $m$-bounded order.
\end{theorem}

We end this short introduction with the comment that also in the case when the order of $\varphi$ is not necessarily two the impact of $C_G(\varphi)$ on the structure of $G$ can be strong (see for example \cite{khukhro}), though not as strong as in the case of involutory automorphisms.

\section{Preliminaries} 

 As usual, $\pi(G)$ denotes the set of prime divisors of the order of $G$. We write $O_\pi(G)$ for the maximal normal $\pi$-subgroup.
 
Given an automorphism $\varphi$ of a group $G$, we denote by $G_{\varphi}$ the centralizer of $\varphi$ in $G$ and by $G_{-\varphi}$ the set $\{g^{-1}g^{\varphi}\mid g\in G\}$. Observe that $[G,\varphi]$ is precisely the subgroup generated by $G_{-\varphi}$.
    
    The following lemma collects some well-known results on involutory automorphisms. In the sequel it will often be used without explicit references. 
    
    \begin{lemma} \label{aut.inv}
        Let $G$ be a finite group of odd order admitting an involutory automorphism $\varphi$. The following properties hold:
        \begin{itemize}
            \item [(i)] $G = G_\varphi G_{-\varphi} = G_{-\varphi} G_\varphi$, and every element $x \in G$ can be written uniquely in the form $x = gh$, where $g \in G_\varphi$ and $h \in G_{-\varphi}$.
        
            \item [(ii)] If $N$ is a normal $\varphi$-invariant subgroup of $G$, we have $(G/N)_{-\varphi} = \left \{ xN \mid x \in G_{-\varphi} \right \}$.
        
            \item [(iii)] If $N$ is a normal $\varphi$-invariant subgroup of $G$ such that $N=N_{\varphi}$ or $N=N_{-\varphi}$, then $\left [ G, \varphi \right ]$ centralizes $N$.
        
            \item [(iv)] If $G=[G,\varphi]$, then the normal closure of $G_\varphi$ is exactly $G'$.  
        
            \item [(v)] $G_\varphi$ normalizes the set $G_{-\varphi}$.
        \end{itemize}
    \end{lemma}
    
Some parts of Lemma \ref{aut.inv} are well known even without assuming that the automorphism has order two. In particular, we have the following lemma (see for example \cite[Theorem 6.2.2]{gorenstein}).
    
    \begin{lemma} \label{aut.cop}
        Let $G$ be a finite group admitting an automorphism $\varphi$ such that $(|\varphi|, |G|)=1$. Then
        \begin{itemize}
            \item [(i)] If $N$ is a normal $\varphi$-invariant subgroup of $G$, then $(G/N)_{\varphi} = G_{\varphi}N/N$.
            
            \item [(ii)] $G = G_{\varphi}[G, \varphi]$.
        
            \item [(iii)] $[G, \varphi] = [G, \varphi, \varphi]$.
        
            \item [(iv)]  For each $p\in \pi(G)$, there is a  $\varphi$-invariant Sylow $p$-subgroup of $G$.
        \end{itemize}
    \end{lemma}
    
    \begin{lemma} \label{aut.cop1}
        Let $G$ be a finite group admitting an automorphism $\varphi$ such that $(|\varphi|, |G|)=1$. If $N$ is a normal $\varphi$-invariant subgroup of $G$ such that $N \leq G_{\varphi}$, then $[G,\varphi]$ centralizes $N$.
    \end{lemma}
    \begin{proof}
        Let $x\in N$ and $g \in G$ be arbitrary elements. Note that $x^{g} \in N$, and therefore $x^{g} = (x^{g})^{\varphi}=x^{g^{\varphi}}$, which implies that $gg^{-\varphi}$ centralizes $x$. Hence the result follows.
    \end{proof}
    
    We will use the above lemmas frequently, sometimes without explicit reference. We also use, without further mention, the Feit-Thompson theorem that every finite group of odd order is solvable \cite{feitthompson}. 

    \begin{lemma} \label{lema1}
        Let $G$ be a finite group of odd order admitting an involutory automorphism $\varphi$ such that $G=[G, \varphi]$. Then $\pi(G_{\varphi})=\pi(G')$. 
    \end{lemma}
    \begin{proof}
        Since $G_{\varphi}\leq G'$, we have $\pi(G_{\varphi}) \subseteq \pi(G')$. Now let $p \notin \pi(G_{\varphi})$. We need to show that $p \notin \pi(G')$. Suppose that this is not the case and $p \in \pi(G')$.  Passing to the quotient $G/O_{p'}(G)$, without loss of generality we may assume that $O_{p'}(G) = 1$.

 Let $M=O_p(G)$. Observe that $C_G(M)\leq M$ (see \cite[Theorem 6.3.2]{gorenstein}). By hypothesis $C_M(\varphi)=1$ and so $M\leq Z(G)$. It follows that $G = M$. Hence $G$ is abelian, and therefore $p\notin \pi(G')$.
    \end{proof}
    
Let $M$ be a normal subgroup of a group $G$. In a natural way $G/C_{G}(M)$ acts on $M$ by conjugation and so we can form the semidirect product $M \rtimes G/C_{G}(M)$. If $G$ admits an automorphism $\varphi$ and $M$ is $\varphi$-invariant, then $\varphi$ induces an automorphism of $M \rtimes G/C_{G}(M)$, which we denote by the same symbol $\varphi$.

    \begin{lemma} \label{pre.lema7}
        Let $G$ be a finite group of odd order admitting an involutory automorphism $\varphi$ such that $G=[G, \varphi]$, and let $M$ be a minimal $\varphi$-invariant normal subgroup of $G$. Assume that $M$ is not contained in $Z(G)$ and let $\tilde{G}=M \rtimes G/C_{G}(M)$. Then $\tilde{G}=[\tilde{G}, \varphi]$. Furthermore, if $M_{\varphi} \cap \gamma_{\infty}(G_{\varphi})=1$, then $M_{\varphi} \cap \gamma_{\infty}(\tilde{G}_{\varphi})=1$.
\end{lemma}
\begin{proof}  Let $\overline{G} = G/C_{G}(M)$. It is easy to check that $M$ is a minimal $\varphi$-invariant normal subgroup of $\tilde{G}$. Indeed, let $1 \neq N \leq M$ be a $\varphi$-invariant normal subgroup of $\tilde{G}$. Since the action of $\overline{G}$ on $N$ coincides with the action of $G$, it follows that $N\trianglelefteq G$. Since $N$ is $\varphi$-invariant, by the minimality of $M$ in $G$ we conclude that $M=N$. 
        
        Next, we show that $\tilde{G}=[\tilde{G}, \varphi]$. Since $G = [G, \varphi]$, it follows that $\overline{G}=[\overline{G}, \varphi]$, which implies that $\overline{G} \leq [\tilde{G}, \varphi]$. Furthermore, since $M$ is not contained in $Z(G)$, we have that $[M, G] \neq 1$. As $[M, \overline{G}]$ is a normal $\varphi$-invariant subgroup contained in M, it follows that $M=[M, \overline{G}]$. Since $[M, \overline{G}] \leq [\tilde{G}, \varphi]$, we obtain $M \leq [\tilde{G}, \varphi]$, which guarantees $\tilde{G}=[\tilde{G}, \varphi]$.
    
        Finally, suppose that $M_{\varphi} \cap \gamma_{\infty}(G_{\varphi}) =1$. Observe that $\tilde{G}_{\varphi}= M_{\varphi} \rtimes \overline{G}_{\varphi}$. Since $[M_{\varphi}, \gamma_{\infty}(G_{\varphi})] \leq M_{\varphi} \cap \gamma_{\infty}(G_{\varphi}) = 1$, we conclude that $\gamma_{\infty}(G_{\varphi})$ acts trivially on $M_{\varphi}$. This means that $\gamma_{\infty}(\overline{G}_{\varphi})$ centralizes $M_{\varphi}$, which ensures that $\gamma_{\infty}(\overline{G}_{\varphi}) \trianglelefteq \tilde{G}_{\varphi}$. Consequently, the quotient $\tilde{G}_{\varphi}/ \gamma_{\infty}(\overline{G}_{\varphi}) \cong M_{\varphi} \rtimes \left(\overline{G}_{\varphi}/ \gamma_{\infty}(\overline{G}_{\varphi})\right)$ is nilpotent. Hence $\gamma_{\infty}(\tilde{G}_{\varphi}) \leq \gamma_{\infty}(\overline{G}_{\varphi})$. By the definition of the semidirect product, we have $M_{\varphi} \cap \overline{G}_{\varphi} = 1$, implying that
        \begin{equation*}
            M_{\varphi} \cap \gamma_{\infty}(\tilde{G}_{\varphi}) \le M_{\varphi} \cap \gamma_{\infty}(\overline{G}_{\varphi})=1.
        \end{equation*}
    \end{proof}

We will require the following proposition from \cite{shumyatsky1993}.
\begin{proposition}\label{93}  Let $G$ be a finite group of odd order admitting an involutory automorphism $\varphi$ such that $G=[G,\varphi]$. Let $N$ be a normal $\varphi$-invariant subgroup of $G$ such that $N_\varphi$ has a normal Hall $\pi$-subgroup $H$. Then $H\leq O_\pi(G)$. In particular, if $N_\varphi$ is nilpotent, then $N_\varphi\leq F(G)$.
\end{proposition}

In particular, the above proposition is used in the proof of the following lemma.

    \begin{lemma}\label{lema7}
        Let $G$ be a finite group of odd order admitting an involutory automorphism $\varphi$ such that $G=[G,\varphi]$. Let $N \leq G'$ be a $\varphi$-invariant normal subgroup of $G$ such that $N_{\varphi}\cap \gamma_{\infty}(G_{\varphi})=1$. Then $N \le Z_{\infty}(G')$.     
 \end{lemma}
    \begin{proof} Suppose the lemma is false and let $G$ be a counterexample of minimal order. It follows from Lemma \ref{aut.inv} that $N/\left\langle N_{\varphi}^{G} \right\rangle$ is contained in the centre of $G/\left\langle N_{\varphi}^{G} \right\rangle$.  Therefore it is sufficient to establish the lemma for the subgroup $\left\langle N_{\varphi}^{G} \right\rangle$ in place of $N$. Thus, without loss of generality, we may assume that $N = \left\langle N_{\varphi}^{G} \right\rangle$. Since $N_\varphi$ is nilpotent, by Proposition \ref{93} we have $N \le F(G)$. Thus, we may assume that $N$ is a $p$-subgroup, for some prime $p \in \pi(N)$. 

Let $M \leq N$ be a minimal normal $\varphi$-invariant subgroup of $G$. By the minimality of $G$, the result holds for the quotient $G/M$. Therefore it is sufficient to show that $M \leq Z_{\infty}(G')$. 

Suppose that $M < C_{G}(M)$. By Lemma \ref{pre.lema7}, the result holds for $\tilde{G} = M \rtimes \overline{G}$, where $\overline{G} = G/C_{G}(M)$, implying that $M \leq Z_{\infty}(\tilde{G}')$. In this case $M \leq Z_{\infty}(G')$, which is a contradiction. Thus, we may assume that $M = C_{G}(M)$ and, consequently, $N = M$.

Let $Q$ be a $\varphi$-invariant  subgroup of $G'$ such that $QN/N$ is a minimal $\varphi$-invariant normal subgroup of $G/N$. Then $Q$ is an elementary abelian $q$-subgroup for some prime $q \neq p$. We have $Q = [Q, \varphi] \times Q_{\varphi}$. 

 Observe that $QN$ is a $\varphi$-invariant normal subgroup and so $(QN)_{\varphi} = Q_{\varphi}N_{\varphi}$. Furthermore, $[Q_{\varphi}, N_{\varphi}] \leq \gamma_{\infty}(G_{\varphi}) \cap N_{\varphi} = 1$. Therefore, $Q_{\varphi}$ centralizes $N_{\varphi}$. It follows that $(QN)_{\varphi}$ is abelian. By Proposition \ref{93}, $Q_{\varphi} \leq O_{q}(G)$ and, consequently, $[N,Q_{\varphi}] =1$. Since $C_G(N) = N$, we deduce that $Q_{\varphi} = 1$, and thus $Q = [Q, \varphi]$. In view of Lemma \ref{aut.inv} (iii) $Q$ centralizes $N / \left\langle N_{\varphi}^{Q} \right\rangle$. Since $N = [N,Q]$, it follows that $N = \left\langle N_{\varphi}^{Q} \right\rangle$. 
 
Since $Q_{\varphi} = 1$, by Lemma \ref{aut.inv} (iii) $QN/N\leq Z(G/N)$. Let $x \in G_{\varphi}$ be a $p'$-element. Note that $[N_{\varphi},x] \leq \gamma_{\infty}(G_{\varphi}) \cap N_{\varphi} = 1$, whence $N_{\varphi} \leq Z(\left\langle N,x \right\rangle)$. Taking into account that $QN/N \leq Z(G/N)$, we deduce that $Q$ normalizes $\left\langle N,x \right\rangle$. Recall that $N = \left\langle N_{\varphi}^{Q} \right\rangle$. 
Therefore $x$ centralizes $N$. Since $C_G(N)=N$ and $x$ is a $p'$-element, we conclude that $x=1$. Since $x$ was chosen arbitrarily, it follows that $G_{\varphi}$ is a $p$-group. By Lemma \ref{lema1},  $G'$ is also a $p$-group. Since $N$ is a minimal normal subgroup contained in $G'$, we obtain that $N \leq Z(G')$, a contradiction. 
\end{proof}

The following result is a theorem of Hartley and Meixner \cite{hartleyandmeixner}

 \begin{theorem} \label{hm} Let $m$ be a positive integer and $G$ a periodic group admitting an involutory automorphism $\varphi$ such that $|C_G(\varphi)|\leq m$. Then $G$ contains a nilpotent of class two subgroup of finite $m$-bounded index.
\end{theorem}

We will also need the following result due to Baer.

    \begin{lemma}[Baer, \cite{baer1953}] \label{teo.baer}
        Let $G$ be a finite group. A $p$-element $x \in G$ belongs to the hypercenter $Z_{\infty}(G)$ if and only if $[x,g]=1$ for every $p'$-element $g \in G$.
    \end{lemma}

The next result is taken from \cite[Theorem 3.2]{kurdachenko}.

  \begin{lemma} \label{kurda}
        Let $G$ be a finite group such that $|G/Z_{\infty}(G)|\leq m$. Then $|\gamma_\infty(G)|$ is $m$-bounded.
    \end{lemma}

%%%%%%%%%%%%%%%%%%%%%%%%%%%%%%%%%%%%%%%%%%%%%%%%%%%%%%%%%%%%%%%%%%%
\section{Proof of the Theorem} 

    We begin this section with an important technical lemma that will be used in the proof of the theorem. Recall that the generalized Fitting subgroup $F^*(G)$ is the product of the Fitting subgroup $F(G)$ and all subnormal quasisimple subgroups; here a group is quasisimple if it is perfect and its quotient by the centre is a nonabelian simple group.   In any finite group $G$ we have $C_G(F^*(G))\leq F^*(G)$ \cite[Theorem 13.12]{hb}. 

\begin{lemma}\label{lema8}
        Let $G$ be a finite group such that $\gamma_{\infty}(G)$ has order $m$ modulo $Z_{\infty}(G)$. Then $\gamma_{\infty}(G)$ has $m$-bounded order.
    \end{lemma} 
    \begin{proof} The proof will proceed by induction on $m$. If $m=1$, then $G$ is nilpotent and we have the result.

Suppose that $m \geq 2$ and the result holds for any group whose nilpotent residual modulo the hypercenter has order strictly less than $m$. Let $\overline{G} = G/Z_{\infty}(G)$. By hypothesis, $|\gamma_{\infty}(\overline{G})| = m$. It follows that the index $[\overline{G} : C_{\overline{G}}(\gamma_{\infty}(\overline{G}))]$ is $m$-bounded. Furthermore, $C_{\overline{G}}(\gamma_{\infty}(\overline{G}))$ is a nilpotent normal subgroup of $m$-bounded index, which implies that the index $[\overline{G}:F(\overline{G})]$ is $m$-bounded. Since $F(\overline{G}) = F(G)/Z_{\infty}(G)$, we obtain that the index $[G : F(G)]$ is bounded in terms of $m$.

If all Sylow subgroups of $F(G)$ are contained in $Z_{\infty}(G)$, then $F(G) = Z_{\infty}(G)$ and $F(\overline{G}) = 1$. In this case $C_{\overline{G}}(F^{*}(\overline{G})) =1$. Since $F^{*}(\overline{G}) = \gamma_{\infty}(F^{*}(\overline{G})) \leq \gamma_{\infty}(\overline{G})$, we have $|F^{*}(\overline{G})| \leq m$. It follows that the order of $\overline{G}$ is at most $m!$. By Lemma \ref{kurda}, we conclude that the order of $\gamma_{\infty}(G)$ is bounded by a function depending only on $m$. 
        
So suppose that there exists a Sylow $p$-subgroup $P$ of $F(G)$ that is not contained in $Z_{\infty}(G)$. By Lemma \ref{teo.baer} there exists a $p'$-element $a \in G$ such that $[P,a] \neq 1$. Write $H = [P,a]$ and $T = H \cap Z_{\infty}(G)$. Observe that $T \leq C_{H}(a)$.

By Lemma \ref{aut.cop1}, $T \leq Z(H)$. Since $G/\gamma_{\infty}(G)$ is nilpotent, elements of coprime orders commute modulo $\gamma_{\infty}(G)$. Since $P$ is a $p$-group and $a$ is a $p'$-element, we have $H = [P,a] \leq \gamma_{\infty}(G)$. Thus, the index $[H : T] \leq m$ and we conclude that $[H : Z(H)]$ is $m$-bounded. By Schur's Theorem  \cite[Theorem~4.12]{rob1}, the order of the derived group $H'$ is $m$-bounded. Note that $H = [P,a] = [H,a]$. Therefore the quotient $H/H'$ has no nontrivial fixed points under the action of $a$. Since $T \leq C_{H}(a)$, the image of $T$ in $H/H'$ must be trivial, whence $T \leq H'$. It follows that the order of $T$ is $m$-bounded, and therefore, the order of $H$ is $m$-bounded as well. 
        
Since $F(G)$ normalizes $H$ and $[G : F(G)]$ is $m$-bounded, the number of conjugates of $H$ in $G$ is bounded in terms of $m$. All conjugates of $H$ normalize each other and have $m$-bounded order. Therefore the order of the normal closure $\left\langle H^{G}\right\rangle$ is bounded in terms of $m$. Set $\tilde{G} = G/\left\langle H^{G}\right\rangle$. Since $H \leq \gamma_{\infty}(G)$, it follows that the order of $\gamma_{\infty}(\tilde{G})Z_{\infty}(\tilde{G})/Z_{\infty}(\tilde{G})$ is strictly less than $m$.

By induction, the order of $\gamma_{\infty}(\tilde{G})$ is bounded in terms of $m$. Since $\gamma_{\infty}(\tilde{G}) = \gamma_{\infty}(G)/\left\langle H^{G}\right\rangle$ and $|\left\langle H^{G}\right\rangle|$ is bounded in terms of $m$, we conclude that the order of $\gamma_{\infty}(G)$ is bounded by a function depending only on $m$. This completes the proof.
    \end{proof} 

    We are now ready to embark on the proof of our main theorem. 

    \begin{proof}[Proof of Theorem \ref{main}]  Recall that $G$ is a finite group of odd order admitting an involutory automorphism $\varphi$ such that $G=[G,\varphi]$. Assume that $\gamma_{\infty}(G_{\varphi})$ has order $m$. We claim that $\gamma_{\infty}(G')$ has $m$-bounded order. In view of Lemma \ref{lema8}, we can pass to the quotient $G/Z_{\infty}(G')$ and assume without loss of generality that $Z(G')=1$.
    
We will proceed by induction on $m$. If $m=1$, then $G_{\varphi}$ is nilpotent. By a theorem of Asar \cite{asar1981}, $G'$ is nilpotent and the result follows immediately. Suppose that $m > 1$ and that the result holds whenever $\gamma_{\infty}(G_{\varphi})$ has order strictly less than $m$.

Assume additionally that every $\varphi$-invariant normal subgroup of $G$ has a nontrivial intersection with $\gamma_{\infty}(G_{\varphi})$.
Let $M$ be a minimal $\varphi$-invariant normal subgroup of $G$. Since the order of $\gamma_{\infty}(G_{\varphi}M/M)$ is strictly less than $m$, it follows by the induction hypothesis that the order of $\gamma_{\infty}(G')/M$ is $m$-bounded.

If there exist two distinct minimal $\varphi$-invariant normal subgroups in $G$, say $M_{1}$ and $M_{2}$, then the orders of $\gamma_{\infty}(G')/M_{1}$ and $\gamma_{\infty}(G')/M_{2}$ are $m$-bounded, which implies that the order of $\gamma_{\infty}(G')$ is $m$-bounded, completing the proof. Thus, we may assume that $M$ is a unique minimal $\varphi$-invariant normal subgroup of $G$. It is sufficient to show that the order of $M$ is $m$-bounded. 

By Lemma \ref{pre.lema7}, the semidirect product $\tilde{G} = M \rtimes G/C_{G}(M)$ satisfies the hypotheses of the theorem. Therefore we may assume that $M=C_{G}(M)$. Clearly, $M\leq Z(F(G))$ and we deduce that $M=F(G)$. Note that since $\gamma_{\infty}(G_{\varphi})$ has order at most $m$, the index $[G_\varphi:F(G_\varphi)]$ is $m$-bounded. By the result from \cite{shumyatsky2007} mentioned in the introduction, the index $[G':F(G')]$ is $m$-bounded, as well. 

Let $\overline{G} = G/M$. The previous paragraph shows that the order of $\overline{G}'$ is $m$-bounded. Observe that $C_{\overline{G}}(\overline{G}')$ is a nilpotent normal $p'$-subgroup of $\overline{G}$. By Lemma \ref{aut.cop}, there exists a nilpotent $\varphi$-invariant $p'$-subgroup $H$ of $G$ such that $\overline{H} = C_{\overline{G}}(\overline{G}')$. Hence, the index $[\overline{G}: \overline{H}]$ is $m$-bounded, which implies that $[G:MH]$ is $m$-bounded. Furthermore, $[H, G_{\varphi}] \le M$ because $G_{\varphi}\leq G'$ and $H$ centralizes $G'$ modulo $M$.

Set $U=O_{p,p'}(G)$. If $U_{\varphi}$ is a $p$-group, then $(U/M)_{\varphi}=1$ and so $U/M\leq Z(\overline{G})$. Since $U/M$ contains its centralizer in $\overline{G}$, we deduce that $U=G$ and so
$G_\varphi$ is a $p$-group, a contradiction since we assumed that $|\gamma_{\infty}(G_{\varphi})|>1$.

Thus, $U_{\varphi}$ is not a $p$-group. Let $L$ be a Hall $p'$-subgroup of $U_{\varphi}$. If $[M_{\varphi}, L]=1$, then, by Proposition \ref{93}, $L\leq O_{p'}(G)=1$, a contradiction. Therefore, there exists some $p'$-element $a \in U_{\varphi}$ such that $[M_{\varphi}, a] \neq 1$. By Lemma \ref{teo.baer}, $M_{\varphi}$ is not contained in $Z_{\infty}(G_{\varphi})$. 

Let $T = [M,a]$ and observe that $T$ is $H$-invariant. Since $M$ admits the decomposition $M = M_{-\varphi}M_{\varphi}$ under the action of $\varphi$, it follows that $T = [M_{-\varphi}, a][M_{\varphi},a]$. Consequently, $T_{\varphi} = [M_{\varphi}, a] \leq \gamma_{\infty}(G_{\varphi})$, showing that the order of $T_{\varphi}$ is at most $m$. Note further that the order of $H_{\varphi}$ is $m$-bounded, since it is equal to the order of $\overline{H}_{\varphi}$. Thus, the order of $(TH)_{\varphi}$ is $m$-bounded and, by  Theorem \ref{hm}, $TH$ contains a nilpotent normal subgroup $K$ of $m$-bounded index. 

Since $T$ is a $p$-group and the quotient $K/T$ is a $p'$-group, it follows that $T \le Z(K)$. Let $x \in T_{\varphi}$ be a nontrivial element. Since $MK \leq C_{G}(x)$, the number of conjugates of $x$ in $G$ is bounded by:
        \begin{equation*}
             [G:MK] \leq [G:MH][TH:K].
        \end{equation*}
Since both factors on the right-hand side are $m$-bounded, we conclude that the conjugacy class $|x^{G}|$ has $m$-bounded size.

 It can be easily checked that the order of the normal closure $\langle x^{G} \rangle$ is $m$-bounded. Indeed, since $x \in M$ and $M$ is elementary abelian, $\langle x^{G} \rangle$ is an elementary abelian $p$-group. Hence, its order is at most $p^{|x^{G}|}$. On the other hand, since $x$ is a nontrivial $p$-element from $\gamma_{\infty}(G_{\varphi})$, we have $p\leq m$. Therefore, $|\langle x^{G} \rangle| \leq m^{|x^{G}|}$. 
        
The subgroup $\langle x^{G} \rangle$ is normal and $\varphi$-invariant. By the minimality of $M$, we must have $M = \langle x^{G} \rangle$, proving that the order of $M$ is $m$-bounded. This completes the proof in the case where every $\varphi$-invariant normal subgroup of $G$ has a nontrivial intersection with $\gamma_{\infty}(G_{\varphi})$.

Now, suppose that there exists a nontrivial $\varphi$-invariant normal subgroup $N\leq G$ such that $N_{\varphi}\cap \gamma_{\infty}(G_{\varphi})=1$. We assume that $N$ is a maximal such subgroup. Then there are no such subgroups in $G/N$ and therefore, by the above, $\gamma_\infty(G')$ has $m$-bounded order modulo $N$. On the other hand, by Lemma \ref{lema7}, $N\cap G'\leq Z_\infty(G')$. Hence, $\gamma_\infty(G')$ has $m$-bounded order modulo $Z_\infty(G')$. Now an application of Lemma \ref{lema8} completes the proof.
    \end{proof}
    
\section{Declarations}

No datasets were generated or analysed during the current study.
\bigskip

Both authors contributed equally to the current study.
\bigskip

The first-named author thanks FAPDF and CNPq for the support during preparation of this paper.

\bibliographystyle{abbrv}
%\bibliography{bibli}

\begin{thebibliography}{b}

\bibitem{asar1981}
    A. O. Asar. \textit{Involutory automorphisms of groups of odd order.} Archiv der Mathematik, 36(1):97--103, 1981. 

\bibitem{baer1953}
    R. Baer. \textit{Group elements of prime power index.} Transactions of the American Mathematical Society, 75(1):20--47, 1953.

\bibitem{feitthompson}
    W. Feit and J. Thompson. \textit{Solvability of groups of odd order.} Pacific Journal of Mathematics, 13:775--1029, 1963.

\bibitem{gorenstein}
    D. Gorenstein. \textit{Finite Groups.} Harper and Row, New York, 1968.

\bibitem{hartleyandmeixner}
    B. Hartley and Th. Meixner. \textit{Periodic groups in which the centralizer of an involution has bounded order.} Journal of Algebra, 64(1):285--291, 1980.

\bibitem{hb}  B. Huppert, N. Blackburn. \textit{Finite Groups III.} Springer-Verlag, 1982.
 
\bibitem{khukhro}  E. I. Khukhro. \textit{Nilpotent Groups and Their Automorphisms.} De Gruyter expositions in mathematics, W. de Gruyter, 1993.
 
\bibitem{kurdachenko}
    L. A. Kurdachenko and I. Ya. Subbotin. \textit{On some properties of the upper and lower central series.} Southeast Asian Bulletin of Mathematics, 37(4):547--554, 2013.

\bibitem{rob1} D. J. S. Robinson, \textit{Finiteness conditions and generalized soluble groups}. Part 1, Springer-Verlag, 1972.

\bibitem{shumyatsky1993}
    P. Shumyatsky. \textit{Involutory automorphisms of locally soluble periodic groups.} Journal of Algebra, 155(1), 36--43, 1993.

\bibitem{ijac96}
    P. Shumyatsky. \textit{Involutory automorphisms of periodic groups.} International Journal of Algebra and Computation, 6, 745--749, 1996.

\bibitem{arch98}
    P. Shumyatsky. \textit{Involutory automorphisms of finite groups and their centralizers.} Arch. Math., 71, 425--432, 1998.

\bibitem{mona05}
    P. Shumyatsky. \textit{Involutory automorphisms of groups of odd order.} Monatsh. Math., 146, 77--82, 2005.

\bibitem{shumyatsky2007}
    P. Shumyatsky. \textit{Centralizers of involutory automorphisms of groups of odd order.} Journal of Algebra, 315(2):954--962, 2007.






    
\end{thebibliography}

\end{document}